\documentclass[11 pt]{amsart}
\usepackage{latexsym,amscd,amssymb, graphicx, amsthm, bm, young, amsbsy}  
\usepackage{arydshln}
\usepackage{tikz}
\usepackage[colorlinks]{hyperref}
\usepackage[margin=1in]{geometry}
\numberwithin{equation}{section}    

\newtheorem{theorem}{Theorem}[section]
\newtheorem{proposition}[theorem]{Proposition}

\theoremstyle{definition}

\newtheorem{remark}[theorem]{Remark}
\newtheorem{example}[theorem]{Example}

\newcommand{\maj}{{\mathrm {maj}}}

\newcommand{\Hilb}{{\mathrm {Hilb}}}
\newcommand{\grFrob}{{\mathrm {grFrob}}}

\newcommand{\SYT}{{\mathrm {SYT}}}

\newcommand{\Frob}{{\mathrm {Frob}}}

\newcommand{\symm}{{\mathfrak{S}}}

\newcommand{\CC}{{\mathbb {C}}}
\newcommand{\QQ}{{\mathbb {Q}}}
\newcommand{\ZZ}{{\mathbb {Z}}}
\newcommand{\NN}{{\mathbb{N}}}

\newcommand{\xx}{{\mathbf {x}}}
\newcommand{\HH}{{\mathbf {H}}}
\newcommand{\LL}{{\mathbf {L}}}
\newcommand{\RR}{{\mathbf {R}}}
\newcommand{\yy}{{\mathbf {y}}}
\newcommand{\II}{{\mathbf {I}}}
\newcommand{\TT}{{\mathbf {T}}}

\begin{document}

\title[Macdonald polynomials and cyclic sieving]
{Macdonald polynomials and cyclic sieving}

\author{Jaeseong Oh}
\address
{Department of Mathematical Sciences \newline \indent
Seoul National University \newline \indent
1 Gwanak-ro, Gwanak-gu,
Seoul, 08826, South Korea}
\email{jaeseong$\_$oh@snu.ac.kr}

\begin{abstract}
The Garsia--Haiman module is a bigraded $\mathfrak{S}_n$-module whose Frobenius image is a Macdonald polynomial. The method of orbit harmonics promotes an $\mathfrak{S}_n$-set $X$ to a graded polynomial ring. The orbit harmonics can be applied to prove cyclic sieving phenomena which is a notion that encapsulates the fixed-point structure of finite cyclic group action on a finite set. By applying this idea to the Garsia--Haiman module, we provide cyclic sieving results regarding the enumeration of matrices that are invariant under certain cyclic row and column rotation and translation of entries.
\end{abstract}

\keywords{Garsia--Haiman module, Kostka polynomial, cyclic sieving, orbit harmonics}
\maketitle

\section{Introduction}
\label{Introduction}

Since Macdonald \cite{Mac88} defined Macdonald polynomials $\widetilde{H}_\mu(\xx;q,t)$ and conjectured the Schur positivity of them, the Macdonald polynomial has been one of the central objects in algebraic combinatorics. Even though a combinatorial formula for Macdonald polynomials is given \cite{HHL04} and the Schur positivity of Macdonald polynomials is proved \cite{Hai01}, not much is known about an explicit combinatorial formula for the $(q,t)$-Kostka polynomials, which are the Schur coefficients of the Macdonald polynomial. In this paper, we discuss some enumerative results involving the $(q,t)$-Kostka polynomials in the words of cyclic sieving phenomena. This will provide a series of identities between the number of matrices with certain cyclic symmetries and evaluation of $(q,t)$-Kostka polynomials at a root of unity, uncovering a part of the mystery of the $(q,t)$-Kostka polynomials.

To begin, we define the cyclic seiving phenomena. Let $X$ be a set with an action of a cyclic group $C$. Fix a generator $c$ of $C$ and let $\zeta$ be a root of unity having the same multiplicative order as $c$. Let $X(q)\in\ZZ[q]$ be a polynomial in $q$. We say that the triple $(X, C, X(q))$ exhibits the \emph{cyclic sieving phenomenon (CSP)} \cite{RSW04} if for any integer $r$ the number of fixed points of $c^r$ in $X$ is equal to the evalation of $X(q)$ at $q=\zeta^r$, i.e.
\begin{equation*}
\left|X^{c^r}\right| = \left|\{ x \in X \,:\, c^r \cdot x = x \}\right| = X(\zeta^r).
\end{equation*}

More generally, let $X$ be a set with an action of a direct product of $k$ cyclic groups $C_1\times C_2\times \cdots \times C_k$. For each $i=1,2,\dots,k$, fix a generator $c_i$ for $C_i$. Let $X(q_1,q_2,\dots,q_k)\in\ZZ[q_1,q_2,\dots,q_k]$ be a polynomial in $k$ variables. Following \cite{BRS08}, we say the triple $\left(X, C_1\times C_2\times \cdots \times C_k, X(q_1,q_2,\dots,q_k)\right)$ exhibits the \emph{k-ary-cyclic sieving phenomenon (k-ari-CSP)} if for any integers $r_1, r_2, \dots, r_k$ the number of fixed points of $(c_1^{r_1},c_2^{r_2},\dots, c_k^{r_k})$ in $X$ is equal to the evalation of $X(q_1,q_2,\dots,q_k)$ at $(q_1,q_2,\dots,q_k)=(\zeta_1^{r_1}, \zeta_2^{r_2},\dots, \zeta_k^{r_k})$, i.e.
\begin{equation*}
\left|X^{(c_1^{r_1},c_2^{r_2},\dots,c_k^{r_k})}\right| = \left|\{ x \in X \,:\, (c_1^{r_1},c_2^{r_2},\dots,c_k^{r_k}) \cdot x = x \}\right| = X(\zeta_1^{r_1}, \zeta_2^{r_2}, \dots, \zeta_k^{r_k}),
\end{equation*}
where $\zeta_i$ is a root of unity having the same multiplicative order as $c_i$. In this paper, we provide instances of tricyclic sieving and quadracyclic sieving phenomena, i.e. $k$-ary-CSP for $k=3$ and $k=4$.

The main tool we used to prove our results is the theory of orbit harmonics. The orbit harmonics is a tool in combinatorial representation theory that promotes an (ungraded) action of $\mathfrak{S}_n$ on a finite set $X$ to a graded action of $\mathfrak{S}_n$ on a polynomial ring quotient, by viewing $X$ as an $\mathfrak{S}_n$-stable point locus in $\mathbb{C}^n$. The idea goes as follows. Let $X\subseteq\mathbb{C}^n$ be a finite set which is closed under the action of $\mathfrak{S}_n\times C$, where
\begin{itemize}
    \item a symmetric group $\mathfrak{S}_n$ acts on $\mathbb{C}^n$ by permuting coordinates, and
    \item $C$ is a finite cyclic group acting on $\mathbb{C}^n$ by scaling a root of unity.
\end{itemize}
Let $\mathbf{I}(X)$ be the ideal of polynomials in $\mathbb{C}[\xx_n]:=\mathbb{C}[x_1,\dots,x_n]$ which vanish on $X$. Then we have an isomorphism
\begin{equation}\label{isomorphism I-ideal}
    \mathbb{C}[X]\cong\mathbb{C}[\xx_n]/\mathbf{I}(X),
\end{equation}
where $\mathbb{C}[X]$ is the algebra of all functions $X\rightarrow \mathbb{C}$. We further define a homogeneous ideal
\begin{equation*}
    \mathbf{T}(X):=\langle\tau(f):f\in\mathbf{I}(X)\setminus\{0\}\rangle\subseteq\mathbb{C}[\xx_n],
\end{equation*}
where $\tau(f)$ denote the top degree homogeneous part of $f$. Then the isomorphism~\eqref{isomorphism I-ideal} extends to an isomorphism 
\begin{equation}\label{isomorphism T-ideal}
    \mathbb{C}[X]\cong\mathbb{C}[\xx_n]/\mathbf{I}(X)\cong\mathbb{C}[\xx_n]/\mathbf{T}(X).
\end{equation}
Note that  $\mathbb{C}[\xx_n]/\mathbf{T}(X)$ admits an additional structure of a graded $\mathfrak{S}_n$-module.

Using the isomorphism~\eqref{isomorphism T-ideal}, the author and Rhoades provided a `generating theorem' for sieving results \cite[Theorem 3.4]{OR20}. By varying the choice of combinatorial locus $X$, one can obtain various sieving results concerning $X$. The associated polynomial is given by a variant of the graded Frobenius image $\grFrob(\CC[\xx_n]/\mathbf{T}(X);q)$.

In this paper, we adopt orbit harmonics to the diagonal orbit harmonics to obtain a `generating theorem' (Theorem~\ref{sieving-generator}) for sieving results of the combinatorial locus $X\subseteq\mathbb{C}^{2n}$ with a diagonal action of $\mathfrak{S}_n$ on $X$. A precise explanation of the diagonal orbit harmonics is given in Section~\ref{subsection Orbit harmonics and cyclic sieving}. 

To prove the Macdonald positivity conjecture, Garsia and Haiman \cite{GH93} suggested a bigraded $\symm_n$-module $\HH_\mu$ (called the Garsia--Haiman module) for a partition $\mu$ whose Frobenius image is the Macdonald polynomial of $\mu$. This module is defined as the $\mathbb{C}$-span of a variant of the Vandermonde determinant and its partial derivatives (See Section~\ref{modules of garsia--Haiman} for detail). Garsia and Haiman defined another module $\RR_\mu$ via orbit harmonics which is later shown to be isomorphic to the original Garsia--Haiman module $\HH_\mu$.  Therefore by taking $\mu=(m^n)$, we can apply Theorem~\ref{sieving-generator} to this module $\RR_\mu$ to obtain triCSP for $n\times m$ matrices of content $(1^{mn})$ (each of $1,2,\dots,mn$ appears once) and biCSP for $n\times m$ matrices of given content $\nu$ ($i$ appears $\nu_i$ times). This relates roots of unity specializations of $(q,t)$-Kostka polynomials with enumerations of matrices invariant under certain rotation of row index and column index and translation of entries.

\begin{theorem}\label{Main theorem 1}
Let $X_{(m^n)}$ be the set of $n\times m$ matrices of content $(1^{mn})$. A product of cyclic groups $\ZZ_n\times\ZZ_m\times\ZZ_{mn}$ acts on $X_{(m^n)}$ by row rotation, column rotation, and adding 1 modulo $mn$ to each entry. In addition for a composition $\nu\models mn$, let $X_{(m^n),\nu}$ be the set of $n\times m$ matrices of content $\nu$ where a product of cyclic groups $\ZZ_n\times\ZZ_m$ acts on $X_{(m^n),\nu}$ by row and column rotation. Then we have the followings.
\begin{itemize}
    \item $\left(X_{(m^n)},\ZZ_n\times\ZZ_m\times\ZZ_{mn},X_{(m^n)}(q,t,z)\right)$ exhibits triCSP, where $$X_{(m^n)}(q,t,z)=\sum_{\lambda\vdash mn} \widetilde{K}_{\lambda,(m^n)}(q,t)f^{\lambda}(z).$$
    \item $\left(X_{(m^n),\nu},\ZZ_n\times\ZZ_m,X_{(m^n),\nu}(q,t)\right)$ exhibits biCSP, where $$X_{(m^n),\nu}(q,t)=\sum_{\lambda\vdash mn} \widetilde{K}_{\lambda,(m^n)}(q,t)K_{\lambda,\nu}.$$
\end{itemize}
Here, $\widetilde{K}_{\lambda,\mu}(q,t)$ ($K_{\lambda,\mu}$, respectively) denotes the modified $(q,t)$-Kostka polynomial (Kostka number, respectively) and $f^{\lambda}(z):=\sum_{T\in{\SYT}(\lambda)}z^{\operatorname{maj}(T)}$ is the fake degree polynomial, where ${\SYT}(\lambda)$ is the set of standard tableaux of shape $\lambda$ and $\operatorname{maj}$ is the major index.
\end{theorem}

% On the other hand, for a partition $\nu$ Garsia--Procesi module $R_\nu$ \cite{GP92} is a finite dimensional graded $\mathfrak{S}_n$ module whose Frobenius image gives modified Hall--Littlewood polynomial $\widetilde{H}_\nu(q;x)$. Morita and Nakajima \cite{MN06} showed $R_\nu$ can also be obtained via orbit harmonics. To be precise, let $\nu$ be a partition of $n$ with $k$ parts. For each word $w=(w_1,\dots,w_n)$ of length $n$ and content $\nu$, associate a point $x_w=(x_1,x_2,\dots,x_n)\in\CC^n$ by the following condition:
% $$x_i=\eta^{w_i}.$$
% where $\eta$ is $k$-th root of unity. Then using orbit harmonics, one can show that $\CC[Y_\nu]\cong R_\nu$.

We generalize Theorem~\ref{Main theorem 1} in two directions. First direction is to generalize biCSP in the second bullet point of Theorem~\ref{Main theorem 1} to triCSP. We say a composition $\nu$ has a cyclic symmetry of order $a$ if $\nu_i=\nu_{i+a}$ always, where the subscripts are interpreted modulo the length $l(\nu)$ of $\nu$. In the second bullet point of the above theorem, if $\nu$ has a cyclic symmetry of order $a$ dividing $l(\nu)$, the set $X_{(m^n),\nu}$ possesses an additional action of a cyclic group $\ZZ_{l(\nu)/a}$ by adding $a$ modulo $l(\nu)$ to each entry. Then one might ask if there is natural $z$-analogue of $X_{(m^n),\nu}(q,t)$ to give triCSP for $X_{(m^n),\nu}$. We give an answer of this question in the following theorem.

\begin{theorem}\label{Main theorem 2}
Let $\nu\models mn$ be a composition with a cyclic symmetry of order $a$ dividing $l(\nu)$. Let $X_{(m^n),\nu}$ be the set of $n\times m$ matrices of content $\nu$. A product of cyclic groups $\ZZ_n\times\ZZ_m\times\ZZ_{l(\nu)/a}$ acts on $X_{(m^n),\nu}$ by row rotation, column rotation, and adding $a$ modulo $l(\nu)$ to each entry. Then the triple $\left(X_{(m^n),\nu},\ZZ_n\times\ZZ_m\times\ZZ_{l(\nu)/a},X_{(m^n),\nu}(q,t,z)\right)$ exhibits the triCSP, where
$$X_{(m^n),\nu}(q,t,z)=\sum_{\lambda\vdash mn} \widetilde{K}_{\lambda,(m^n)}(q,t)\widetilde{K}_{\lambda,\nu}(z).$$
Here, $\widetilde{K}_{\lambda,\mu}(q,t)$ ($\widetilde{K}_{\lambda,\mu}(z)$, respectively) denotes the modified $(q,t)$-Kostka polynomial ($z$-Kostka polynomial, respectively).
\end{theorem}

The second generalization starts from rephrasing Theorem~\ref{Main theorem 1}. For a $mn \times mn$ permutation matrix $M$, we associate a $n \times m$ matrix $\phi(M)$ of content $(1^{mn})$ by letting $k$ be the $(i,j)$-entry if the $(n(i-1)+j, k)$-entry of $M$ is 1. For example, if we set $m=2, n=2$, then the following permutation matrix $M$ in the left corresponds to a $2 \times 2$ matrix $\phi(M)$ in the right.
\begin{figure}[h]
  $M=\begin{pmatrix}
0 & 0 & 1 & 0 \\
1 & 0 & 0 & 0 \\
0 & 1 & 0 & 0 \\
0 & 0 & 0 & 1
  \end{pmatrix}$
  \qquad$\phi(M)=\begin{pmatrix}
  3 & 1\\
  2 & 4
  \end{pmatrix}$
  \end{figure}\\
Under this correspondence, the row rotation, column rotation, and adding 1 modulo $mn$ to each entry of $\phi(M)$ corresponds to `external' row rotation (of order $m$), `internal' row rotation (of order $n$), and column rotation of $M$. Here, by external row rotation, we mean sending the $k$th $n$ rows (from $n(k-1)+1$-th to $nk$-th row) to the next $n$ rows (from $nk+1$-th to $n(k+1)$-th row). Here the row numbers are interpreted modulo $mn$. By internal row rotation, we mean sending each row to the next row except for the $nk$-th row for $k=1,2,\dots,m$. For the $nk$-th row we send this row to the $n(k-1)+1$-th row. In our running example, if we apply a row rotation and a column rotation to $\phi(M)$, we get $\begin{pmatrix} 2&4 \\ 3&1\end{pmatrix}$ and $\begin{pmatrix} 1&3 \\ 4&2\end{pmatrix}$. Each of these corresponds to the matrix 
\[
\begin{pmatrix}
0&1&0&0\\0&0&0&1 \\ 0&0&1&0\\ 1&0&0&0 \end{pmatrix} 
\text{ and } 
\begin{pmatrix}
1&0&0&0\\0&0&1&0 \\ 0&0&0&1\\ 0&1&0&0 \end{pmatrix}.
\] 
The first matrix can also be obtained from $M$ by sending the first two rows to the third and the fourth row and sending the third and the fourth row to the first and the second row which is an external row rotation (of order 2). Similarly, the second matrix can be obtained from $M$ by applying an internal row rotation (of order 2). Now we can understand Theorem~\ref{Main theorem 1} as a tricyclic sieving result concerning $mn \times mn$ matrices under external row rotation, internal row rotation, and column rotation. One might ask if we could get a quadracyclic sieving result concerning external and internal rotation to both columns and rows. The following theorem gives a positive answer.

\begin{theorem}\label{Main theorem 3}
Let $l=mn=ab$ be a positive integer with two factorizations. Let $\mathfrak{S}_{l}$ be the set of $l\times l$ permutation matrices. A product of cyclic groups $\ZZ_n\times\ZZ_m\times\ZZ_{b}\times \ZZ_{a}$ acts on $\mathfrak{S}_l$ by external row rotation, internal row rotation, external column rotation, and internal column rotation. Then the triple $\left(\mathfrak{S}_l, \ZZ_n\times\ZZ_m\times\ZZ_{b}\times \ZZ_{a}, \mathfrak{S}_l(q,t,z,w)\right)$ exhibits the quadraCSP, where

$$\mathfrak{S}_l(q,t,z,w)=\sum_{\lambda\vdash l} \widetilde{K}_{\lambda,(m^n)}(q,t)\widetilde{K}_{\lambda,(a^b)}(z,w).$$

Here, $\widetilde{K}_{\lambda,\mu}(q,t)$ denotes the modified $(q,t)$-Kostka polynomial.
\end{theorem}

There have been similar results discovered which relate root of unity specializations of $q$-Kostka polynomials or Macdonald polynomials and fixed point enumerations of matrices or fillings of tableaux (see \cite{Rho10, AU19, BRS08} for example). We remark some results which are especially close to our results. First of all, the set $X_{(m^n),\nu}$ in Theorem~\ref{Main theorem 2} bijects with the set of $0,1$-matrices with column content $\mu$ and row content $(1^{mn})$. If we specialize $n=1$, then Theorem~\ref{Main theorem 2} recovers \cite[Theorem 1.2]{Rho10}. In addition, Barcelo, Reiner and Stanton considered biCSP concerning row and column rotation of permutation matrices (\cite[Theorem 1.4]{BRS08}, in the case $W=\mathfrak{S}_n$). Theorem~\ref{Main theorem 3} specializes to their result if we take $n=1, b=1$.

In \cite[Theorem 1.3]{Rho10}, using Hall--Littlewood polynomial, Rhoades showed that $\NN$-matrices with fixed column content $\mu$ and row content $\nu$ exhibits biCSP. It should be mentioned that we modify the argument of Rhoades to prove Theorem~\ref{Main theorem 2} and Theorem~\ref{Main theorem 3} in Section~\ref{subsection: proof of the main theorem 2}. 

The remainder of this paper is organized as follows. In  Section~\ref{Preliminaries} we give background on combinatorics, symmetric functions, the representation theory of $\symm_n$, and the Garsia--Haiman modules. In Section~\ref{Section: Sieving generating theorem and orbit harmonics}, we illustrate the diagonal orbit harmonics and how to obtain sieving results from orbit harmonics. We then present a combinatorial locus that gives a graded module isomorphic to Garsia--Haiman module via orbit harmonics. In Section~\ref{Section: proofs}, we provide proofs of Theorem~\ref{Main theorem 1}, Theorem~\ref{Main theorem 2} and Theorem~\ref{Main theorem 3}. We conclude this paper with some remarks in Section~\ref{concluding remarks}.

\section{Preliminaries}
\label{Preliminaries}

\subsection{Combinatorics}\label{subsection Combinatorics}
A \emph{weak composition} of $n$ is a sequence of non-negative integers which sum to $n$. A \emph{composition} is a weak composition which consists of positive integers. A \emph{partition} of $n$ is a composition of $n$ which is weakly decreasing. We denote $\mu\models n$ and $\lambda\vdash n$ for a composition $\mu$ and a partition $\lambda$ of $n$. For a composition $\mu=(\mu_1,\dots,\mu_k)$, the \emph{length} $l(\mu)$ of $\mu$ is $k$.

For a partition $\lambda\vdash n$ we abbuse our notation so that a partition $\lambda$ also denotes for its \emph{Young diagram}. We draw Young diagrams in French style:
$$\lambda=\{(i,j)\in\ZZ_{\ge0}\times\ZZ_{\ge0}:i<\lambda_{j+1}\}.$$ The elements of Young diagram are called cells. For example,
\begin{align*}
    \lambda &= (4,3,1)\\
            &=\{(0,0),(1,0),(2,0),(3,0),(0,1),(1,1),(2,1),(0,2)\}
\end{align*}

\quad\qquad\qquad\qquad\qquad = \quad\begin{young}
   \cr
 & & \cr
 & & & \cr
\end{young}\\
We define the \emph{conjugate} $\lambda'$ to be the partition obtained by reflecting $\lambda$ with respect to the diagonal line $x=y$ in the plane. 

A \emph{tableau} of a partition $\lambda$ is a filling $T:\lambda\rightarrow\ZZ_{>0}$. In this case, we call the \emph{shape} of $T$ is $\lambda$. The \emph{content} of a tableau $T$ of $\lambda$ is a weak composition $(T_1,T_2,\dots)$ of $n$, where $T_i$ is the number of $i$'s appearing in $T$. A tableau $T$ is called \emph{semistandard} if the entries in each row are weakly increasing (left to right) and the entries in each column are strictly increasing (bottom to top). The \emph{Kostka number} $K_{\lambda,\mu}$ is the number of semistandard tableaux of shape $\lambda$ and content $\mu$. A semistandard tableau is called \emph{standard} if its content is $(1,1,\dots)$. The set of standard tableaux of shape $\lambda$ is denoted by $\SYT(\lambda)$. Examples of semistandard tableau and standard tableau of shape $(4,3,1)$ are shown in the left and the right below, respectively.

\begin{center}
\begin{young}
 4  \cr
 2& 2& 4\cr
 1& 1& 2& 3\cr
\end{young}\qquad \qquad
\begin{young}
 5  \cr
 2 & $\mathbf{4}$ & 7\cr
 $\mathbf{1}$ & $\mathbf{3}$ & $\mathbf{6}$ & 8\cr
\end{young}

\end{center}

For a standard tableau $T$, a \emph{descent} is an index $i$ such that $i+1$ is in the upper row than $i$. The \emph{major index} $\maj(T)$ of $T$ is defined as the sum of all descents in $T$. For the standard tableau given in the right above, $1, 3, 4$ and $6$ are descents (descents of the tableau are written in bold), so the major index of the tableau is $1+3+4+6=14$. The \emph{fake degree polynomial} of a partition $\lambda$ is defined by the major index generating function for the standard tableaux of shape $\lambda$, i.e.,
\begin{equation*}
    f^\lambda(q):=\sum_{T\in \SYT(\lambda)}q^{\maj(T)}.
\end{equation*}
\subsection{Symmetric functions}
\label{subsection Symmetric functions}
Let $\Lambda=\bigoplus_{d\ge0}\Lambda_d$ be the ring of symmetric functions in an infinite number of variables $\xx=(x_1,x_2,\dots)$ over $\mathbb{C}(q,t)$. Here $\Lambda_d$ denotes the subspace of $\Lambda$ consisting of symmetric functions of homogeneous degree $d$. 

Bases of $\Lambda_n$ are indexed by partitions $\lambda\vdash n$. Among various bases of $\Lambda_n$, one of the most important basis is given by Schur functions. The \emph{Schur function} $s_\lambda$ of a partition $\lambda$ is defined by
\begin{equation*}
    s_\lambda(\xx):=\sum_{T}\xx^T,
\end{equation*}
where the sum is over all semistandard tableaux of shape $\lambda$ and $\xx^T=x_1^{T_1}x_2^{T_2}\cdots$.

We write $\langle \cdot, \cdot\rangle$ for the Hall inner product 
\begin{equation*}
    \langle s_\lambda, s_\mu \rangle =\delta_{\lambda,\mu},
\end{equation*}
where $\delta_{\lambda,\mu}$ is the Kronecker delta. The (modified) \emph{Macdonald polynomials} $\widetilde{H}_\lambda(\xx;q,t)$ indexed by partitions $\lambda\vdash n$ form another basis of $\Lambda_n$. They are defined by the unique family satisfying the following \emph{triangulation} and \emph{normalization} axioms \cite{HHL04},
\begin{itemize}
    \item $\widetilde{H}_\lambda[\xx(1-q);q,t]=\sum_{\lambda\ge\mu} a_{\lambda,\mu}(q,t)s_\lambda$,
    \item $\widetilde{H}_\lambda[\xx(1-t);q,t]=\sum_{\lambda\ge\mu'} b_{\lambda,\mu}(q,t)s_\lambda$, 
    \item $\langle \widetilde{H}_\mu, s_{(n)}\rangle=1$,
\end{itemize}
for suitable coefficients $a_{\lambda,\mu}, b_{\lambda,\mu'}\in\QQ(q,t)$. Here, a partial order $\le$ called \emph{dominance order} of partitions of $n$ is defined by 
\begin{equation*}
    \lambda\le\mu \text{ if } \lambda_1+\cdots+\lambda_k\le\mu_1+\cdots\mu_k \text{ for all } k,
\end{equation*} and $[\cdot]$ denotes the plethystic substitution. These axioms are equivalent with Macdonald's triangularity and orthogonality axioms.

Expanding the Macdonald polynomial with Schur functions, we may write 
\begin{equation*}
    \widetilde{H}_\mu(\xx;q,t)=\sum_\lambda \widetilde{K}_{\lambda,\mu}(q,t)s_{\lambda}(\xx),
\end{equation*}
where the sum is over partitions $\lambda$ of $n$. The coefficients $\widetilde{K}_{\lambda,\mu}(q,t)$ are called the (modified) $(q,t)$-\emph{Kostka polynomials}. A combinatorial description of general $(q,t)$-Kostka polynomials is unknown, and it is one of the most important open problems in algebraic combinatorics. Since $\widetilde{K}_{\lambda,\mu}(1,1)=|\SYT(\lambda)|$, the most desirable form of the combinatorial formula would be a generating function for the standard tabelaux.

The (modified) \emph{Hall--Littlewood polynomial} $\widetilde{Q}_\mu(X;q)$ and $q$-\emph{Kostka polynomial} $\widetilde{K}_{\lambda,\mu}(q)$ can be obtained by specializing $t=0$ to the Macdonald polynomial $\widetilde{H}_\mu(X;q,t)$ and the $(q,t)$-Kostka polynomial $\widetilde{K}_{\lambda,\mu}(q,t)$. The $q$-Kostka polynomial $\widetilde{K}_{\lambda,\mu}(q)$ can also be defined as the generating function of the cocharge statistics for the semistandard tableaux of shape $\lambda$ and content $\mu$ (see \cite{Rho10} for a definition of cocharge statistics). 

\subsection{Representation theory of $\mathfrak{S}_n$}\label{subsection rep of Sn}

Irreducible representations of the symmetric group $\mathfrak{S}_n$ are in one to one correspondence with partitions $\lambda$ of $n$. We let $S^\lambda$ be the corresponding irreducible representation. If $V$ is a finite dimensional $\mathfrak{S}_n$-module, there is a unique way of decomposing $V$ into irreducibles as $V=\bigoplus_{\lambda\vdash n} c_\lambda S^\lambda$. The \emph{Frobenius image} of $V$ is the symmetric function defined by 
\begin{equation*}
    \Frob(V):=\sum_{\lambda\vdash n} c_\lambda s_\lambda.
\end{equation*}
If $V$ is graded (or bigraded) $\mathfrak{S}_n$-module as $V=\bigoplus_{d\ge0} V_d$ (or $V=\bigoplus_{d,e\ge0} V_{d,e}$) the graded Frobenius image is the symmetric function over $\CC(q)$ (or $\CC(q,t)$) given by

\begin{equation*}
\grFrob(V;q) :=\sum_{d\ge0}\Frob(V_d)q^d
\end{equation*}
\begin{equation*}
    (\text{or }\grFrob(V;q,t):=\sum_{d,e\ge0} \Frob(V_{d,e})q^d t^e).
\end{equation*}
If $V = \bigoplus_{d \geq 0} V_d$ (or $V=\bigoplus_{d,e\ge 0} V_{d,e}$ is any graded (or bigraded) vector space, its {\em Hilbert series} is
 \begin{equation*}
 \Hilb(V;q) = \sum_{d \geq 0} \dim (V_d) q^d
 \end{equation*}
 \begin{equation*}
    (\text{or }\Hilb(V;q,t):=\sum_{d,e\ge0} \dim(V_{d,e})q^d t^e).
\end{equation*}

We recall two facts that we use in the proof of the main results. The first one is the theorem of Springer. Note that the original result of Springer deals with the action of a regular element of an arbitrary complex reflection group. For simplicity, we focus only on the action of a long cycle in a symmetric group $\symm_n$.

\begin{theorem}(\cite{Spr74})
\label{springer-theorem}
Let $c=(1,2,\dots,n)$ be a long cycle of $\symm_n$ and $\chi^\lambda:\symm_n \rightarrow \CC$ be the irreducible character associated to the $\mathfrak{S}_n$-representation $S^\lambda$. Then we have
\begin{equation*}
\chi^\lambda(c^r) = f^{\lambda}(\zeta^r),
\end{equation*}
where $f^\lambda(q) \in \CC[q]$ is the fake degree polynomial and $\zeta$ is a (primitive) $n$-th root of unity.
\end{theorem}

The second fact we need is about the Kronecker coefficients. For $\symm_n$-modulues $V$ and $W$, define the \emph{inner tensor product} $V\otimes W$ to be the the usual tensor product of vector spaces with $\mathfrak{S}_n$-module structure given by $$\sigma\cdot(v\otimes w)=(\sigma \cdot v)\otimes (\sigma \cdot w).$$ For given partitions $\lambda, \mu$ and $\nu$ of $n$, the \emph{Kronecker coefficients} $g^\lambda_{\mu,\nu}$ is the multiplicity of $S^\lambda$ in the inner tensor product $S^\mu\otimes S^\nu$, i.e. $$S^\mu\otimes S^\nu\cong\bigoplus g^\lambda_{\mu,\nu}S^\lambda.$$ Although giving an explicit combinatorial description of general Kronecker coefficients is difficult in general (it is one of the major open problems in algebraic combinatorics), we have the following identity for the special case when $\lambda$ is a single row.

\begin{proposition}\label{prop:Kronecker} For partitions $\mu,\nu$ of $n$, the Kronecker coefficient
$g^{(n)}_{\mu,\nu}=\delta_{\mu,\nu},$
where $\delta_{\mu,\nu}$ is the Kronecker delta.
\end{proposition}

\subsection{Modules of Garsia and Haiman}\label{modules of garsia--Haiman}

To each partition $\mu$ of $n$, let $(a_1, b_1), \dots, (a_n, b_n)$ be the cells of $\mu$, taken in some arbitrary order. Then we define
\[
\Delta_\mu:=\mathrm{det}(x_i^{a_j} y_i^{b_j})_{1\le i,j\le n}.
\]
The $\symm_n$-module $\mathbf{H}_\mu$ is the smallest vector space over $\CC$ that contains the determinant $\Delta_\mu$ and its partial derivatives with respect to any of the variables $x_i$'s and $y_i$'s for $1\le i \le n$. For example, for a partition $\mu=(3,2)$, the cells of $\mu$ are $(0,0), (1,0), (2,0), (0,1)$ and $(1,1)$ and the corresponding determinant is given by 
\begin{equation*}
\Delta_\mu:=\mathrm{det}\left[\begin{array}{ccccc}
1 & x_1 & x_{1}^{2} & y_1 & x_{1} y_1 \\
1 & x_2 & x_{2}^{2} & y_2 & x_{3} y_3 \\
1 & x_3 & x_{3}^{2} & y_3 & x_{3}y_3 \\
1 & x_4 & x_4^2 & y_4 & x_{4}y_4 \\
1 & x_5 & x_5^{2} & y_5 & x_{5}y_5
\end{array}\right].
\end{equation*}
Then the module $\mathbf{H}_\mu$ is given by the $\mathbb{C}$-span 
\begin{equation*}
\mathbb{C}\{\partial_{\xx_I}\partial_{\yy_J}\Delta_\mu\}_{I,J}
%=\mathbb{C}\{\Delta_\mu, \partial_{x_1} \Delta_\mu, \partial_{x_2} \Delta_\mu, \partial_{y_1}\Delta_\mu, \partial_{y_2}\Delta_\mu, 1 \},
\end{equation*}
where $\partial_{\xx_I}:=\partial_{x_{i_1}}\dots\partial_{x_{i_k}}$ for a multiset $I=\{i_1,\cdots,i_k\}$ and $\partial_{\yy_J}$ is defined similarly. The symmetric group $\symm_n$ acts on $\HH_\mu$ diagonally i.e. permuting $x$ and $y$ coordinates simultaneously. Since the action of $\symm_n$ preserves the degree of $\xx_n$ and $\yy_n$, the module $\HH_\mu$ is bigraded $\symm_n$-module, where the bigrade is given by degree of $\xx_n$ and $\yy_n$. This is called the \emph{Garsia--Haiman module}. Haiman \cite{Hai01} proved the \emph{$n!$ conjecture} which asserts that this module is of dimension $n!$ regardless of $\mu$, and moreover, the graded Frobenius image of $\HH_\mu$ is the Macdonald polynomial of $\mu$:
\begin{equation}\label{graded Frob=Macdonald}
    \grFrob(\HH_\mu;q,t)=\widetilde{H}_\mu(\xx;q,t)=\sum_\lambda \widetilde{K}_{\lambda,\mu}(q,t)s_\lambda(\xx),
\end{equation}
which proves the Schur positivity of the Macdonald polynomials.

\section{Sieving generating theorem and diagonal orbit harmonics}
\label{Section: Sieving generating theorem and orbit harmonics}

\subsection{Orbit harmonics and cyclic sieving}
\label{subsection Orbit harmonics and cyclic sieving}
In this section, we introduce a systematic way to generate sieving results using orbit harmonics. The author and Rhoades provided a `generating theorem' for sieving results \cite[Theorem 3.4]{OR20} by exploiting orbit harmonics applied to a locus $X\subseteq \mathbb{C}^n$ with $\symm_n$ acting on $X$ by permuting coordinates. To modify this idea for our purpose, we first explain the \emph{diagonal orbit harmonics} (see \cite{GH96} for more details). Consider a finite set $X\subseteq \CC^{2n}$ which is closed under $\mathfrak{S}_n\times C_1\times C_2$-action where 
\begin{itemize}
    \item a symmetric group $\mathfrak{S}_n$ acts on $\mathbb{C}^{2n}$ by permuting coordinates diagonally, i.e. $$\sigma(x_1,\dots,x_n,y_1,\dots,y_n)=(x_{\sigma(1)},\dots,x_{\sigma(n)},y_{\sigma(1)},\dots,y_{\sigma(n)}),$$
    \item a finite cyclic group $C_1$ is acts on $\mathbb{C}^n$ by scaling $x$-coordinates by a root of unity, and
    \item a finite cyclic group $C_2$ is acts on $\mathbb{C}^n$ by scaling $y$-coordinates by a root of unity.
\end{itemize}
Then the method of orbit harmonics gives us an isomorphism of $\mathfrak{S}_n\times C_1 \times C_2$-modules:
\begin{equation}\label{isomorphism T-ideal xy}
    \mathbb{C}[X]\cong\mathbb{C}[\xx_n,\yy_n]/\mathbf{I}(X).
\end{equation}
We further define homogeneous ideal
\begin{equation*}
    \TT(X):=\langle \tau_x\circ\tau_y(f): f\in \II(X)\setminus \{0\}\rangle \subseteq \CC[\xx_n],
\end{equation*}
where $\tau_x$ and $\tau_y$ is the map taking top degree homogeneous part with respect to $\xx_n$ and $\yy_n$, respectively. Then the isomorphism~\eqref{isomorphism T-ideal xy} extends to an isomorphism
\begin{equation*}
    \mathbb{C}[X]\cong\mathbb{C}[\xx_n,\yy_n]/\mathbf{I}(X)\cong\mathbb{C}[\xx_n,\yy_n]/\mathbf{T}(X),
\end{equation*}
where the last item $\mathbb{C}[\xx_n,\yy_n]/\mathbf{T}(X)$ has an additional structure of graded $\mathfrak{S}_n\times C_1\times C_2$-module on which $C_1$ and $C_2$ acts by scaling in each fixed (bi)degree. Thanks to this isomorphism, we can provide a generating theorem for sieving results in diagonal orbit harmonics whose proof is analogous to the proof of Theorem 3.4 in \cite{OR20}.

\begin{theorem}
\label{sieving-generator}
Let $C$ be the subgroup of $\mathfrak{S}_n$ generated by a long cycle $c=(1,2,\dots,n)$. Fix positive integers $k_1$ and $k_2$. For $j=1,2$, let
$\zeta_j := \exp(2 \pi i / k_j) \in \CC^{\times}$ and $C_j = \langle c_j \rangle \cong \ZZ_{k_j}$ be a cyclic group of order $k_j$. Consider the action of $\mathfrak{S}_n \times C_1\times C_2$ on $\CC^{2n}$ where $c_1$ scales $x$-coordinates by $\zeta_1$, $c_2$ scales $y$-coordinates by $\zeta_2$ and $\mathfrak{S}_n$ acts by permuting coordinates diagonally.

Let $X \subseteq \CC^{2n}$ be a finite point set which is closed under the action of $\mathfrak{S}_n\times C_1\times C_2$.
\begin{enumerate}
\item  Suppose that for $d,e \geq 0$, the isomorphism type of the degree $(d,e)$-piece of 
$\CC[\xx_n,\yy_n]/\TT(X)$ is given by
\begin{equation*}
( \CC[\xx_n,\yy_n]/\TT(X) )_{d,e} \cong \bigoplus_{\lambda\vdash n} c_{\lambda,d,e}S^\lambda.
\end{equation*}
The triple $(X, C_1\times C_2 \times C, X(q,t,z))$ exhibits the tricyclic sieving phenomenon where 
\begin{equation*}
X(q,t,z) = \sum_{\lambda\vdash n}  c_{\lambda}(q,t) f^{\lambda}(z).
\end{equation*}
where $c_{\lambda}(q,t) := \sum_{d,e \geq 0} c_{\lambda,d,e} q^d t^e$.
\item Let $G \subseteq \mathfrak{S}_n$ be a subgroup. The set $X/G$ of $G$-orbits in $X$ carries a natural
$C_1\times C_2$-action and the triple $(X/G, C_1\times C_2, X/G(q,t))$ exhibits the bicyclic sieving phenomenon where
\begin{equation*}
X/G(q,t) = \Hilb( (\CC[\xx_n,\yy_n]/\TT(X))^G; q,t).
\end{equation*}
\end{enumerate}
\end{theorem}

\begin{proof} Applying orbit harmonics to the action of $\symm_n \times C_1 \times C_2$ on $X$ 
yields an isomorphism of ungraded $\symm_n \times C_1 \times C_2$-modules
\begin{equation}
\label{ungraded-isomorphism}
\CC[X] \cong \CC[\xx_n,\yy_n]/\TT(X).
\end{equation}
Let $\zeta := \exp(2 \pi i / n)$.
To prove (1), apply Theorem~\ref{springer-theorem} to obtain that 
for any integers $r, s, k$, the trace of $(c_1^r, c_2^s, c^k) \in C_1\times C_2 \times C'$ acting on $\CC[\xx_n,\yy_n]/\TT(X)$ is given by $$\sum_{\lambda\vdash n} c_\lambda(\zeta_1^r, \zeta_2^s) f^{\lambda}(\zeta^k) = X(\zeta^r, \zeta_2^s, \zeta^k).$$
By the isomorphism \eqref{ungraded-isomorphism}, this coincides with the trace of $(c_1^r, c_2^s, c^k)$ which is the number of fixed
points of $(c_1^r, c_2^s, c^k)$ acting on $X$, completing the proof of (1).

For (2), we take $G$-invariants of both sides of the isomorphism \eqref{ungraded-isomorphism}
to get an isomorphism of $C_1\times C_2$-modules
\begin{equation}
\label{g-invariants}
\CC[X/G] \cong (\CC[\xx_n,\yy_n]/\TT(X))^G.
\end{equation}
Since $C_1\times C_2$ acts on the graded vector space $(\CC[\xx_n,\yy_n]/\TT(X))^G$ by a root of unity scaling for each $x$ and $y$ variables, the trace of $(c_1^r, c_2^s)$ on the right hand side of the isomorphism \eqref{g-invariants} is given by
$$[\Hilb( (\CC[\xx_n,\yy_n]/\TT(X))^G; q,t)]_{q = \zeta_1^r, t=\zeta_2^s} = X/G(\zeta_1^r,\zeta_2^s).$$ 
The trace of $(c_1^r, c_2^s)$ on the left hand side of the isomorphism \eqref{g-invariants} coincides with the number 
of $G$-orbits in $X/G$ fixed by $(c_1^r, c_2^s)$.
\end{proof}

\begin{remark}\label{remark sieving generator}
In order to obtain a sieving result involving a combinatorial set $X$ with a cyclic group action using Theorem~\ref{sieving-generator}, we must 
\begin{itemize}
    \item realize $X$ (or its quotient $X/G$) and the relevant action on it as a point locus in $\CC^{2n}$ and the compatible action,
    \item calculate the graded Frobenius image of $\CC[\xx_n,\yy_n]/\TT(X)$ or the Hilbert series of the quotient $\Hilb\left((\CC[\xx_n,\yy_n]/\TT(X))^G; q,t\right)$.
\end{itemize}

\end{remark}

\subsection{Orbit harmonics and Garsia--Haiman module}\label{Modules of Garsia--Haiman and Garsia--Procesi}
There is a way to understand the Garsia--Haiman module $\HH_\mu$ via orbit harmonics. Let $\mu$ be a partition of $n$ with $l(\mu)=l$ and $l(\mu')=l'$. Let $\{\alpha_0,\dots,\alpha_{l-1}\}$ and $\{\beta_0,\dots,\beta_{l'-1}\}$ be two sets of distinct complex numbers. Recall that a \emph{injective tableau} $T$ of shape $\mu\vdash n$ is a filling of cells of $\mu$ by integers $1,2,\dots,n$ without repetition. The collection of such tableaux will be denoted by $\mathbf{IT}(\mu)$. For each $T\in\mathbf{IT}(\mu)$, we assign a point $p_T\in\CC^{2n}$ by letting the $i$-th and the $(n+i)$-th coordinates of $p_T$ record the position of $i$ in $T$:
\begin{equation*}
    p_T=(\alpha_{y_T(1)},\dots,\alpha_{y_T(n)},\beta_{x_T(1)},\dots,\beta_{x_T(n)}),
\end{equation*}
where $x_T(i)$ and $y_T(i)$ are $x$ and $y$ coordinates of the cell which contains $i$ in $T$. For example, for a partition $\mu=(2,1)$ and an injective tableau $T=$\begin{young}
 2 \cr
 3& 1\cr
\end{young} of shape $\mu$, the point assigned for $T$ is $p_T=(\alpha_0, \alpha_1, \alpha_0, \beta_1, \beta_0, \beta_0)$. Let us denote the collection of points associated to the injective tableaux by
\begin{equation*}
    X_\mu=\{p_T\in\CC^{2n}:T\in\mathbf{IT}(\mu)\}.
\end{equation*}
Note that there are exactly $n!$ points in $X_\mu$. The point locus $X_\mu$ possesses a natural diagonal action of $\symm_n$: For $\sigma\in\symm_n$,
\begin{equation*}
    \sigma (x_1,\dots,x_n,y_1,\dots,y_n)=(x_{\sigma(1)},\dots,x_{\sigma(n)},y_{\sigma(1)},\dots,y_{\sigma(n)}).
\end{equation*}
Using orbit harmonics, one can promote the ungraded $\symm_n$-module $\CC[X_\mu]$ to the bigraded $\symm_n$-module. As usual, let $\mathbf{I}(X_\mu)$ be the ideal of polynomials in $\mathbb{C}[\xx_n,\yy_n]$ which vanish on $X$ and define a homogeneous ideal 
\begin{equation*}
    \mathbf{T}(X_\mu):=\langle\tau_x\circ\tau_y(f):f\in\mathbf{I}(X_\mu)\setminus\{0\}\rangle\subseteq\mathbb{C}[\xx_n,\yy_n].
\end{equation*}
Then the module $\RR_\mu:=\mathbb{C}[\xx_n,\yy_n]/\mathbf{T}(X_\mu)$ has an additional structure of (bi)graded $\mathfrak{S}_n$-module.

Garsia and Haiman \cite{GH96} proved that the Garsia--Haiman module $\HH_\mu$ embedds into this graded module $\RR_\mu$. Thanks to the $n!$-conjecture, we can conclude the following isomorphism between $\RR_\mu$ and $\HH_\mu$.
\begin{theorem}
\label{theorem:Garsia--Haiman module} We have an isomorphism as bigraded $\symm_n$-modules:
$$\RR_\mu\cong\CC\left[X_\mu\right]\cong \HH_\mu.$$
\end{theorem}

\section{Proofs of main theorems}
\label{Section: proofs}
\subsection{A proof of Theorem~\ref{Main theorem 1}}
\label{subsection: proof of the main theorem}

We first construct a point locus $X_{(m^n)}$ associated to a rectangular partition $\mu=(m^n)$. Following Section~\ref{Modules of Garsia--Haiman and Garsia--Procesi}, to consider $X_{(m^n)}$ as a point locus in $\mathbb{C}^{2n}$, we choose two sets of distinct complex numbers $\{\alpha_0,\dots,\alpha_{n-1}\}$ and $\{\beta_0,\dots,\beta_{m-1}\}$. For our purpose, let $\zeta_1=\exp(\frac{2\pi i}{n})$ and $\zeta_2=\exp(\frac{2\pi i}{m})$, then set $\alpha_j=\zeta_1^{j}$ for $0\le j \le n-1$ and $\beta_k=\zeta_2^{k}$ for $0\le k \le m-1$. Then the corresponding locus $X_\mu$ possesses
\begin{itemize}
    \item diagonal action of $\symm_n$,
    \item action of a cyclic group $C_1$ of order $n$ acting by scaling a root of unity $\zeta_1$ to each $x$-coordinates, and 
    \item action of a cyclic group $C_2$ of order $m$ acting by scaling a root of unity $\zeta_2$ to each $y$-coordinates.
\end{itemize}

Now we can present a proof of Theorem~\ref{Main theorem 1}. By the construction above, $X_{(m^n)}$ has $\symm_n\times C_1\times C_2$ action which corresponds to permutation of letters, row rotation, and column rotation on $\mathbf{IT}(\mu)$, respectively. Combining an isomorphism between $\RR_\mu$ and $\HH_\mu$ (Theorem~\ref{theorem:Garsia--Haiman module}), the fraded Frobenius image of $\HH_\mu$ (Equation~\eqref{graded Frob=Macdonald}) and the sieving generating theorem (Theorem~\ref{sieving-generator}), the first bullet point of Theorem~\ref{Main theorem 1} immediately follows.

To proceed to the second bullet point, consider a composition $\nu$ of $mn$. For the Young subgroup $G=\symm_\nu=\symm_{\nu_1}\times\symm_{\nu_2}\cdots$ of $\nu$, the $G$-orbits of $X_{(m^n)}$ are in one-to-one correspondence with the set of $n\times m$ matrices with content equal to $\nu$. 

To obtain a sieving result for $X_{(m^n)}/G$, we must calculate the Hilbert series of $G$-fixed subspace of $\RR_\mu$. Let ${\bf 1}$ be the trivial representation of $\symm_\nu$. It is a standard fact that the induction of ${\bf 1}$ from $\symm_\nu$ to $\symm_n$
can be written as
\begin{equation*}
{\bf 1} \uparrow_{\symm_\nu}^{\symm_n} \cong \bigoplus_\lambda K_{\lambda,\nu} S^\lambda,
\end{equation*}
where $K_{\lambda,\nu}$ denotes a Kostka number. Applying Frobenius reciprocity, it follows that the dimension of the $\symm_\nu$-fixed subspace of 
the $\symm_n$-irreducible $S^{\lambda}$ is given by the character inner product:
\begin{equation*}
\dim (S^{\lambda})^{\symm_\nu} = \langle {\bf 1}, S^{\lambda} \downarrow^{\symm_n}_{\symm_\nu} \rangle_{\symm_\nu} = 
\langle {\bf 1} \uparrow_{\symm_\nu}^{\symm_n}, S^{\lambda} \rangle_{\symm_n} = K_{\lambda,\nu}.
\end{equation*}
Correspondingly, if $V$ is any bigraded $\symm_n$-module with Frobenius image
\begin{equation*}
    \grFrob(V;q,t)=\sum_{\lambda\vdash n}c_{\lambda}(q,t)S^\lambda,
\end{equation*}
the Hilbert series of $\symm_\nu$ fixed subspace will be
\begin{equation*}
    \Hilb(V^{\symm_\nu};q,t)=\sum_{\lambda\vdash n} c_{\lambda}(q,t)K_{\lambda,\nu}.
\end{equation*}
By (2) of Theorem~\ref{sieving-generator}, this concludes the second bullet point.

\subsection{A proof of Theorem~\ref{Main theorem 2} and Theorem~\ref{Main theorem 3}}
\label{subsection: proof of the main theorem 2}

In previous section, we proved that for a composition $\nu$ of $mn$, the triple $\left(X_{(m^n),\nu}, \ZZ_n\times \ZZ_m, \sum_\lambda \widetilde{K}_{\lambda,\mu}(q,t)K_{\lambda,\nu}\right)$ exhibits biCSP. Suppose, furthermore, $\nu$ has a cyclic symmetry of order $a$. Then the set $X_{(m^n),\nu}$ possesses another cyclic group action by adding $a$ modulo $l(\nu)$ to each entry. Therefore, it is natural to seek for a sieving result that reflects this additional cyclic group action. 

%The bigraded module $H_\mu$ can be viewed as $\mathfrak{S}_{|\mu|}\times \ZZ_{l_1}\times \ZZ_{l_2}$ module by the cyclic group $\ZZ_{l_1}$ and $\ZZ_{l_2}$ act on graded component $H_\mu^{d_1,d_2}$ by scaling $\exp(\frac{2\pi i d_1}{l_1})$ on $x$-part and $\exp(\frac{2\pi i d_2}{l_2})$ on $y$-part.

Before we begin, we recall the Tanisaki locus. For a composition $\nu\models d$, let $W_\nu$ be the set of length $d$ words $w=(w_1,\dots,w_d)$ of content $\nu$ ($i$ appears $\nu_i$ times). Let $\zeta=\exp\left(\frac{2\pi i}{l(\nu)}\right)$. We assign a point $p_w$ in $\CC^d$ so that we can realize $W_\nu$ as a point locus $Y_\nu$ (called the Tanisaki locus) in $\CC^d$ as follows:
\begin{equation*}
    p_w=(\zeta^{w_1},\dots,\zeta^{w_d}).
\end{equation*}
Garsia and Procesi \cite{GP92} proved that the T-ideal corresponding to the Tanisaki locus is given by the ideal generated by elementary symmetric polynomials with extra conditions (for precise definition of this `Tanisaki ideal', we refer \cite{GP92}). By orbit harmonics, there is an isomorphism
\begin{equation}\label{Garsia--Procesi module orbit harmonics}
    \CC[Y_\nu]\cong \LL_\nu:=\CC[\xx_d]/\TT(Y_\nu).
\end{equation}
Moreover, they showed that the graded Frobenius image coincides with the Hall--Littlewood symmetric function, $$\grFrob(\LL_\nu;q)=\widetilde{Q}_\nu(\xx;q).$$
Furthermore, if a composition $\nu$ has a cyclic symmetry of order $a$, the set $W_\nu$ has additional cyclic group action given by adding $a$ modulo $l(\nu)$ to each letter. This action corresponds with the action of scaling a root of unity $\zeta^a$ in each coordinates in $Y_\nu$. In this setting, the isomorphism~\eqref{Garsia--Procesi module orbit harmonics} extends to an isomorphism as graded $\symm_d\times C$-modules, where $C$ is a cyclic group of order $l(\nu)/a$.

Now let $\mu=(m^n)$ be a rectangular partition and $\nu$ be a composition of $mn$ with a cyclic symmetry of order $a$. Then the product $X_\mu\times Y_\nu$ carries an $\mathfrak{S}_{mn}\times \ZZ_n\times \ZZ_m \times \ZZ_{l(\nu)/a}$-action, where $\mathfrak{S}_{mn}$ acts diagonally on $X_\mu$ and $Y_\nu$ and the cyclic groups $\ZZ_n$, $\ZZ_m$ and $\ZZ_{l(\nu)/a}$ acts by row rotation on $X_\mu$, column rotation on $X_\mu$ and translation on the entries on $Y_\nu$, respectively. By Theorem~\ref{theorem:Garsia--Haiman module} and the isomorphism~\eqref{Garsia--Procesi module orbit harmonics}, we have an isomorphism
\begin{equation}\label{equation: isomorphism between product}
\CC[X_\mu\times Y_\nu]\cong \RR_\mu \otimes \LL_\nu
\end{equation}
as $\mathfrak{S}_{mn}\times\ZZ_n\times\ZZ_m\times\ZZ_{l(\nu)/a}$-modules. Since the graded Frobenius image of the module $\RR_\mu$ is given by the Macdonald polynomial and the graded Frobenius image of the Garsia--Procesi module $\LL_\nu$ is given by the Hall--Littlewood polynomial, the Frobenius image is given by
\begin{equation*}
\label{Frob Xmu times Ynu}
    \grFrob\left(\CC[X_\mu\times Y_\nu];q,t,z\right)=\sum_{\rho,\lambda,\lambda'\vdash mn} \widetilde{K}_{\lambda,\mu}(q,t)\widetilde{K}_{\lambda',\nu}(z)g^{\rho}_{\lambda,\lambda'}s_\rho,
\end{equation*}
where $g^\rho_{\lambda,\lambda'}$ denotes a Kronecker coefficient.
By taking isotypic components for a trivial representation $S^{(mn)}$ of $\mathfrak{S}_{mn}$ on both sides of equation~\eqref{equation: isomorphism between product}, we have
\begin{equation}\label{equation: isomorphism isotypic components}
\CC[X_\mu\times Y_\nu]^{\symm_{mn}}\cong [\RR_\mu\otimes \LL_\nu]^{\symm_{mn}}.
\end{equation}

There is a natural basis of $\CC[X_\mu\times Y_\nu]^{\mathfrak{S}_{mn}}$ indexed by $\mathfrak{S}_{mn}$-orbits of $X_\mu\times Y_\nu$, given by the sum of elements in each orbit. Note that each of these orbits corresponds to a $n\times m$ matrix with content equal to $\nu$. It is clear that the cyclic groups $\ZZ_n$, $\ZZ_{m}$ and $\ZZ_{l(\nu)/a}$ act on these matrices by row rotation, column rotation, and translation of the entries. For an element $g\in \ZZ_n\times \ZZ_m \times \ZZ_{l(\nu)/a}$, we can count the number of fixed points of $g$ in $(X_\mu\times Y_\nu)/{\mathfrak{S}_{mn}}$ is given by the trace of $g$ acting on the left hand side of the isomorphism~\eqref{equation: isomorphism isotypic components}. On the other hand, this can be calculated by trigraded Hilbert series 
\begin{equation}\label{Hilbert poly of S_mn invariant of R_mu times L_nu}
    \Hilb\left(\left[\RR_\mu\otimes \LL_\nu\right]^{\symm_{mn}};q,t,z\right)=\sum_{\lambda, \lambda'\vdash mn} \widetilde{K}_{\lambda,(m^n)}(q,t)\widetilde{K}_{\lambda',\nu}(z)g^{(mn)}_{\lambda,\lambda'},
\end{equation}
of $[\RR_\mu\otimes_{\CC}\LL_\nu]^{\mathfrak{S}_{mn}}$ at roots of unity. By Proposition~\ref{prop:Kronecker}, taking the coefficient of the Schur function $s_{(mn)}$ in Equation~\eqref{Hilbert poly of S_mn invariant of R_mu times L_nu}, we have the following polynomial
$$ \Hilb\left([\RR_\mu\otimes \LL_\nu]^{\symm_{mn}};q,t,z\right)=X_{\mu,\nu}(q,t,z):=\sum_{\lambda\vdash mn} \widetilde{K}_{\lambda,(m^n)}(q,t)\widetilde{K}_{\lambda,\nu}(z)$$
for sieving result. This proves Theorem~\ref{Main theorem 2}.
\begin{example}
Take $\mu=(2,2)$ and $\nu=(2,2)$. We have six $2\times2$ matrices with content equal to $\nu$ listed in the following.

\begin{figure}[h]
  $
    \begin{pmatrix}
    1 & 1 \\
    2 & 2
    \end{pmatrix}$
    \qquad$
    \begin{pmatrix}
    1 & 2\\
    1 & 2
    \end{pmatrix}$
    \qquad$
    \begin{pmatrix}
    1 & 2\\
    2 & 1
    \end{pmatrix}$
    \qquad$
    \begin{pmatrix}
    2 & 1\\
    1 & 2
    \end{pmatrix}$
    \qquad$
    \begin{pmatrix}
    2 & 1\\
    2 & 1
    \end{pmatrix}$
    \qquad$
    \begin{pmatrix}
    2 & 2\\
    1 & 1
    \end{pmatrix}$

  \end{figure}

Fixed points of $(0,1,1)\in\ZZ_2\times\ZZ_2\times\ZZ_2$ correspond to the following four matrices and there is no fixed point of $(0,0,1)\in\ZZ_2\times\ZZ_2\times\ZZ_2$

\begin{figure}[h]
  $
    \begin{pmatrix}
    1 & 2\\
    1 & 2
    \end{pmatrix}$
    \qquad$
    \begin{pmatrix}
    1 & 2\\
    2 & 1
    \end{pmatrix}$
    \qquad$
    \begin{pmatrix}
    2 & 1\\
    1 & 2
    \end{pmatrix}$
    \qquad$
    \begin{pmatrix}
    2 & 1\\
    2 & 1
    \end{pmatrix}$
  \end{figure}

The polynomial $X(q,t,z)$ is given by
\begin{align*}
X(q,t,z)
    &=\sum_{\lambda\vdash 4} \widetilde{K}_{\lambda,(2,2)}(q,t)\widetilde{K}_{\lambda,(2,2)}(z)\\
    &=\widetilde{K}_{(4),(2,2)}(q,t)\widetilde{K}_{(4),(2,2)}(z)+\widetilde{K}_{(3,1),(2,2)}(q,t)\widetilde{K}_{(3,1),(2,2)}(z)+\widetilde{K}_{(2,2),(2,2)}(q,t)\widetilde{K}_{(2,2),(2,2)}(z)\\
    &\equiv 3+qz+tz+qtz \qquad \operatorname{mod} \quad(q^2-1, t^2-1, z^2 -1).
\end{align*}
Note that $X(-1,1,-1)=4$ and $X(1,1,-1)=0$, which agrees with Theorem~\ref{Main theorem 2}.
\end{example}

To keep this paper concise, instead of providing a precise proof of Theorem~\ref{Main theorem 3}, we rather sketch a proof that is very similar to the proof of Theorem~\ref{Main theorem 2}. Let $l=mn=ab$ be a positive integer with two factorizations. Following the argument in the proof of Theorem~\ref{Main theorem 2}, we consider point locus $X_{(m^n)}$ to the $l\times l$ permutation matrices by the map $\phi$ defined in Section~\ref{Introduction}. We also consider $X_{(a^b)}$ as the set of $l\times l$ permutation matrices in the same way.

A product of groups $\symm_l\times \mathbb{Z}_n\times \mathbb{Z}_m$ acts on $X_{(m^n)}$ by permuting columns, external row rotation and internal row rotation, while a product of groups $\symm_l \times\mathbb{Z}_b\times \mathbb{Z}_a$ acts on $X_{(a^b)}$ by permuting columns, external row rotation and internal column rotation. Each $\mathfrak{S}_l$-orbits of $X_{(m^n)}\times X_{(a^b)}$ corresponds to a $l\times l$ permutation matrix. One can see that there is a natural action of $\mathbb{Z}_n\times \mathbb{Z}_m\times\mathbb{Z}_b\times \mathbb{Z}_a$ on orbits in $[X_{(m^n)}\times X_{(a^b)}]/\mathfrak{S}_l$ by external and internal row rotation, and external and internal column rotation.

On the other hand, let $\RR_{(m^n)}$ and $\RR_{(a^b)}$ be two bigraded ring obtained from point loci $X_{(m^n)}$ and $X_{(a^b)}$ by orbit harmonics. Since their graded Frobenius images are Macdonald polynomials, by Proposition~\ref{prop:Kronecker}, we have
\begin{align*}\label{Hilbert poly of S_mn invariant of R_mu times R_nu}
    \Hilb\left(\left[\RR_{(m^n)}\otimes \RR_{(a^b)}\right]^{\symm_{l}};q,t,z,w\right)&=\sum_{\lambda, \lambda'\vdash l} \widetilde{K}_{\lambda,(m^n)}(q,t)\widetilde{K}_{\lambda',(a^b)}(z,w)g^{(mn)}_{\lambda,\lambda'}\\
    &=\sum_{\lambda\vdash l} \widetilde{K}_{\lambda,(m^n)}(q,t)\widetilde{K}_{\lambda,(a^b)}(z,w).
\end{align*}
By applying orbit harmonics, we have an isomorphism
\[
\mathbb{C}\left[X_{(m^n)}\times X_{(a^b)}\right]^{\mathfrak{S}_l}\cong \left[\RR_{(m^n)}\otimes \RR_{(a^b)}\right]^{\symm_{l}}.
\]
Therefore, we can conclude that the triple $\left(\mathfrak{S}_l, \ZZ_n\times\ZZ_m\times\ZZ_{b}\times \ZZ_{a}, \mathfrak{S}_l(q,t,z,w)\right)$ exhibits the quadracyclic sieving phenomenon, where
$$\mathfrak{S}_l(q,t,z,w)=\sum_{\lambda\vdash l} \widetilde{K}_{\lambda,(m^n)}(q,t)\widetilde{K}_{\lambda,(a^b)}(z,w).$$

\section{Concluding remarks}\label{concluding remarks}

\subsection{Other combinatorial loci}
Recall that we obtained Theorem~\ref{Main theorem 2} and Theorem~\ref{Main theorem 3} by taking the tensor product of modules of Garsia--Haiman and Garsia--Procesi, and then taking $\symm_{mn}$-invariant part. We could replace one of those modules to obtain various sieving results. One way to do this is replacing one of the modules with the module $\RR_{n,k}$ defined in \cite{HRS18}. They defined this module to construct a graded $\symm_n$-module for the Delta conjecture at $t=0$. The module $\RR_{n,k}$ can also be obtained by applying orbit harmonics to the locus corresponding to the set of surjective functions from $[k]$ to $[n]$. For $mn\le k$ by taking $\symm_{mn}$ invariant part of $\RR_{(m^n)}\otimes \RR_{mn,k}$, we could obtain a triCSP for $n$ times $m$ matrices (or fillings of a rectangular partition) with entries given by nonempty set partitions of $[k]$ into $[mn]$ parts (which may be called the surjective tableaux).

This process can be applied to a broad class of modules obtained via orbit harmonics. One of the interesting modules which we did not consider in this paper is the module $\RR_{\mu,k}$ defined by Griffin \cite{Gri21}. This module is a common generalization of the Garsia--Procesi module and the module $\RR_{n,k}$ of Haglund--Rhoades--Shimozono and it is possible to obtain this module via orbit harmonics. 

\subsection{Combinatorial proof of main theorems}
Rhoades used two facts to prove cyclic sieving results involving $q$-Kostka polynomials \cite{Rho10}. The first one is about the evaluation of the Hall--Littlewood polynomial at a root of unity due to Lascoux, Leclerc and Thibbon \cite{LLT94, LLT97}. The second one is the rim hook correspondence of Stanton and White \cite{SW85}.

For the evaluation of the Macdonald polynomials of a rectangular partition at a root of unity, Descouens, and Morita gave a formula \cite{DM08}. If one can provide a counterpart for rim hook correspondence of Stanton--White in the setting of Theorem~\ref{Main theorem 1}, Theorem~\ref{Main theorem 2} or Theorem~\ref{Main theorem 3}, it would give more combinatorial proof of those.

\section{Acknowledgements}

The author is grateful to Brendon Rhoades for helpful conversations about orbit harmonics. The author also thanks anonymous referees for their careful reading and valuable comments. In particular, the author thanks a referee's suggestion to clearly spell out the connections and differences between the results in this paper and known CSP.

\bibliographystyle{alpha}

\end{document}